\documentclass[10pt,a4paper]{amsart}
\usepackage{utopia}
\usepackage{ulem}
\usepackage{amssymb}
\usepackage{amsthm}
\usepackage{mathrsfs}
\usepackage{eufrak}
\usepackage{graphicx}
\usepackage{ stmaryrd }
\usepackage{float}
\usepackage{amsmath}
\usepackage[all]{xy}\usepackage{xypic}
\usepackage{hyperref}
\hypersetup{%
  pdftitle   = {},
  pdfauthor  = {},
  pdfcreator = {\LaTeX\ with package \flqq hyperref\frqq}
}

\DeclareMathAlphabet{\mathpzc}{OT1}{pzc}{m}{it}

\newtheorem{theorem}{Theorem}[section]
\newtheorem*{theorem*}{Theorem}

\newtheorem{lemma}[theorem]{Lemma}
\newtheorem*{lemma*}{Lemma}
\newtheorem{corollary}[theorem]{Corollary}

\theoremstyle{definition}
\newtheorem{definition}[theorem]{Definition}
\newtheorem{example}[theorem]{Example}
\newtheorem{notation}[theorem]{Notation}
\theoremstyle{remark}
\newtheorem{remark}[theorem]{Remark}

\DeclareMathOperator{\Ext}{Ext}
\DeclareMathOperator{\End}{End}

\DeclareMathOperator{\Hom}{Hom}
\DeclareMathOperator{\Pic}{Pic}
\DeclareMathOperator{\Isom}{Isom}
\DeclareMathOperator{\Aut}{Aut}

\DeclareMathOperator{\coker}{coker}
\DeclareMathOperator{\Tot}{Tot}

\DeclareMathOperator{\ua}{\underline{a}}
\DeclareMathOperator{\ub}{\underline{b}}
\DeclareMathOperator{\uc}{\underline{c}}
\DeclareMathOperator{\ud}{\underline{d}}

\numberwithin{equation}{section}

\begin{document}
\title{Moduli Stacks of Bundles on Local Surfaces}

\author{Oren Ben-Bassat and Elizabeth Gasparim}

\address{Oren Ben-Bassat, Department of Mathematics, University of Haifa, Mount Carmel
 31905 Haifa, Israel}
\email{oren.benbassat@gmail.com}

\address{Elizabeth Gasparim, Imecc - Unicamp
13083-970 Campinas, SP, Brasil}

\address{Elizabeth Gasparim, School of Mathematics, University of Edinburgh, James Clerk Maxwell building,
The King's buildings, Mayfield Road, Edinburgh EH9 3JZ, Scotland}
\email{etgasparim@gmail.com}



\thanks{We would like to thank Edoardo Ballico, Andrew Kresch and Tony Pantev for helpful conversations we had while working on this project.  We would like to thank the Universities of Haifa and Edinburgh, Marco Andreatta, Fabrizio Catanese, the University of Trento and Fondazione Bruno Kessler,  the Isaac Newton Institute for Mathematical Sciences 2011 VBAC Conference, and the 2011 Mirror Symmetry and Tropical Geometry Conference in Cetraro, Italy for travel support.}

\begin{abstract} We give an explicit groupoid presentation of certain stacks of vector bundles on formal neighborhoods of rational curves inside algebraic surfaces.  The presentation involves a 
M\"obius type action of an automorphism group on a space of extensions.
 \end{abstract}

\maketitle
\tableofcontents
\section{Introduction}

A fundamental question is algebraic geometry is to understand how rational maps on a variety X affect the moduli of vector bundles on X, that is: suppose $X$ and $Y$
birationally equivalent, then what is the relation between the various 
moduli of vector
bundles on $X$ and $Y$?
 Here we focus on  the case of 
surfaces, in which case
 rational maps are obtained by blowing up (possibly singular)
points. Suppose $\pi\colon
Y \rightarrow X$ is the blow up of a point $x$ in $X$, with 
$\ell= \pi^{-1}(x)$.
Considering pullbacks, one can then study relative situation of the moduli of vector bundles on $X$ mapping into the moduli of vector bundles on $Y$. Since 
$\pi$ is an isomorphism outside $\ell$ clearly the heart of the question
lies in the geometry of moduli of bundles on a small neighborhood of $\ell$. 
This question was addressed from the point of view of moduli spaces of equivalence classes of vector bundles
in \cite{ADV} for the case when $x$ is a smooth point, and the geometry of 
the local moduli was used to prove the Atiyah--Jones conjecture for rational 
surfaces. In this paper we consider the moduli {\it stacks} of vector bundles in formal neighborhoods of $\ell$, and give explicit groupoid presentations of such moduli stacks. The stacky point of view, besides clarifying several delicate issues about the local moduli also has the advantage that it generalises to the case of singular surfaces, where $\ell$ is a line with self-intersection $\ell^2 = -k <-1.$  We develop the study of stacks of bundles on (completions of the) local surfaces
$Z_k=\Tot({\mathcal O}(-k))$ and give presentations of certain stacks of rank 2 bundles over these surfaces.  The most interesting aspect of these presentations is the "M\"obius" transformation (\ref{nobulletmobius}) discussed in \ref{Mob}.
\section{Local surfaces and vector bundles on them}

\begin{notation}  In this paper we will work with (associative, commutative, unital) $\mathbb{C}$-algebras.  Therefore, affine scheme will mean the spectrum of such an algebra, and all varieties, schemes, and formal schemes are considered over $\mathbb{C}$.  We will  work over the site of affine schemes or $\mathbb{C}$-algebras with the faithfully flat topology.  The schemes we will consider are quasi-compact and quasi-separated.   
For any {\it positive} integer $k$, we have the algebraic variety \[Z_k = \Tot(\mathcal O_{\mathbb P^1}(-k))=Spec_{\mathbb{P}^{1}} \left(\bigoplus_{i=0}^{\infty} \mathcal{O}_{\mathbb{P}^{1}}(ik)\right) \] and $\ell \cong \mathbb{P}^{1}$ its zero section,
so that $\ell^2 = -k$. Let $I_{\ell}$ be the sheaf of $\mathcal{O}_{Z_{k}}$ ideals defining $\ell$.  We write
$Z_{k}^{(n)}$
for the $n^\mathrm{th}$ infinitesimal neighborhood of $\ell$
 and
$\widetilde{Z_{k}}= Z_{k}^{(\infty)}$ for the ``formal'' neighborhood of $\ell$ in $Z_k$. To avoid any confusion, we should mention that $\widetilde{Z_{k}}= Z^{(\infty)}_k$ is a standard scheme and not a formal one, even though we sometimes call it the formal neighborhood of $\ell$.   It is a (an inductive or direct) limit in the category of schemes of the directed system $Z_{k}^{(\bullet)}$:
\[
\mathbb{P}^{1} = Z_{k}^{(0)}\to Z_{k}^{(1)} \to Z_{k}^{(2)} \to \cdots.
\]  The scheme $\widetilde{Z_{k}}= Z_{k}^{(\infty)}$ is defined as the relative spectrum over $Z_{k}$ of the quasicoherent sheaf of  $\mathcal{O}_{Z_{k}}$ algebras
\[\text{lim}_{n} \mathcal{O}_{Z_{k}}/I_{\ell}^{n} =\prod_{i=0}^{\infty} \mathcal{O}_{Z_{k}}(ik)
.\]
Let $\widehat{Z_{k}}= (\ell, \text{lim}_{n} \mathcal{O}_{Z_{k}}/I_{\ell}^{n})$ denote the formal scheme given as the formal completion of $Z_{k}$ along $\ell$.    It is a (an inductive or direct) limit in the category of ringed spaces over $\mathbb{P}^{1}$.  The formal scheme $\widehat{Z_{k}}$ can also be regarded as the formal completion of $\widetilde{Z_{k}}$ along $\ell$.  On the other hand  $\widetilde{Z_{k}}$ can be recovered from $\widehat{Z_{k}}$ as the relative spectrum over $\ell$ of the quasicoherent sheaf of  $\mathcal{O}_{\ell}$ algebras
\[\text{lim}_{n} \mathcal{O}_{\widehat{Z_{k}}}/I_{\ell}^{n} =\prod_{i=0}^{\infty} \mathcal{O}_{\ell}(ik)
.\]
There is a presentation
\[Z_k  = \left(U \bigsqcup V\right)/\sim,\]  where we will always use the charts
$U =   {\mathbb C}^2$ with coordinates $(z,u)$, and
$V =   {\mathbb C}^2$  with coordinates $(\xi,v)$, with
$U \cap V = ({\mathbb  C} - \{0\})  \times   {\mathbb C}$
where the equivalence relation $\sim$ is given by the change of coordinates
$ (\xi,v) = (z^{-1 },z^ku)$.
 Note that the zero section $\ell$
is given in these coordinates  by $u=0$ in the $U$-chart
and $v=0$ in the $V$-chart.  It is easy to see that $I_{\ell} \cong \mathcal{O}(k)$.  In fact, $I_{\ell}$ is the line bundle associated to the divisor $-\ell$ and since $u=\xi^{k}v$,
\[div(u) = \ell + k f
\]
where $f$ is the fiber defined by $\xi=0$.
We similarly have \[U^{(n)} = Spec(\mathbb{C}[z,u]/(u^{n+1}))\] and \[V^{(n)} = Spec(\mathbb{C}[\xi,v]/(v^{n+1}))\] and the associated "limits" \[\widetilde{U}= U^{(\infty)}=Spec(\mathbb{C}[z][[u]])\] and  \[\widetilde{V}= V^{(\infty)}=Spec(\mathbb{C}[\xi][[v]]).\]  As above, we have
\[Z^{(n)}_k  = \left(U^{(n)} \bigsqcup V^{(n)}\right)/\sim,\]
\[\widetilde{Z_k}  = \left(\widetilde{U} \bigsqcup \widetilde{V}\right)/\sim.
\]
and
\[\widehat{Z_k}  = \left(\widehat{U} \bigsqcup \widehat{V}\right)/\sim
\]
where $\widehat{U}$ and $\widehat{V}$ are the formal scheme completions of $U$ and $V$ along $\ell$.
\begin{remark}Unless we explicitly state that $n$ is finite, in each usage of the spaces $Z_{k}^{(n)}$ we are including the case that $n=\infty$.
\end{remark}
These presentations are helpful for describing vector bundles.  For instance by the answer to Serre's famous question (proved by Seshadri  \cite{Se} 
and in further generality by Quillen \cite{Qi} and Suslin \cite{Su}), $U = Spec(\mathbb{C}[z,u])$ has no non-trivial vector bundles; similarly this is true for $U^{(n)}$ and $\widehat{U}$ by Theorem 7 of \cite{Co}.
All the schemes we have mentioned up until now are Noetherian and $\widehat{Z_{k}}$ is a Noetherian formal scheme.  If $T$ is any affine scheme (Noetherian or not) then there is an equivalence of groupoids of vector bundles
\[\text{Vect}(\widetilde{Z_{k}} \times T) \cong \text{Vect}(\widehat{Z_{k}} \times T)
\]
which preserves the tensor structure and takes $\mathcal{O}_{\widetilde{Z_{k}}}$ to $\mathcal{O}_{\widehat{Z_{k}}}$.    In the case that $T$ is a point this just follows from a simple application of Theorem 9.7 from II.9 of \cite{HA}.  Also notice that $\mathcal{O}(\widetilde{Z_{k}})=\mathcal{O}(\widehat{Z_{k}})$.  
If $T$ is an affine scheme such that $\text{Pic}(T)$ is trivial then
\[\Pic(\widehat{Z_k} \times T)\simeq \Pic(\widetilde{Z_k} \times T) \simeq
\Pic(Z^{(n)}_k \times T) \simeq \Pic({\mathbb P}^1 \times T) \simeq \Pic(\mathbb{P}^{1})\simeq \mathbb{Z};
\] we will use the
symbol ${\mathcal O}(j)$ for the line bundle with first Chern class $j$
coming from $\mathbb{P}^{1}$ in any of these spaces.   If $E$ is a rank $2$ vector bundle of first Chern class zero on $Z^{(n)}_k$ then the splitting type $j\geq 0$ of $E$ is the integer such that the restriction of $E$ to $\ell $ is isomorphic to $\mathcal{O}(j) \oplus \mathcal{O}(-j)$.  For a vector bundle on $Z^{(n)}_k \times T$ we say that it has constant splitting type $j$ if its splitting type is $j$ over every $t \in T(\mathbb{C})$.
\end{notation}

For our explicit presentations of stacks, we will need the following basic
results about rank 2 bundles on $Z^{(n)}_k$.
\begin{lemma}\label{reps}
Let $S$ be any scheme over $\mathbb{C}$ and $E$ a rank $2$ vector bundle on $Z^{(n)}_k \times S$ of constant splitting type $j \geq 0$.   Then for any $s \in S(\mathbb{C})$ there is an open subscheme $T$ of $S$ containing $s$ and such that the restriction of $E$ to $Z^{(n)}_k \times T$ has the structure of an extension
\[0 \rightarrow {\mathcal O}(-j) \rightarrow E|_{Z^{(n)}_k \times T}
\rightarrow \mathcal{O}(j) \rightarrow 0.
\]
\end{lemma}
\begin{proof}
By \cite{CA1} Theorem 3.3. $E|_{Z^{(n)}_k \times \{s\}}$ can be written  as an algebraic extension
$$0 \rightarrow {\mathcal O}(-j) \rightarrow E|_{Z^{(n)}_k \times \{s\}}
\rightarrow {\mathcal O}(j) \rightarrow 0$$  where $j >0$.  Consider the leftmost injective map as a nowhere vanishing element of the space of global sections $H^{0}(Z_{k}^{(n)} \times \{s\}, E|_{Z_{k}^{(n)} \times \{s\}} \otimes \mathcal{O}(j))$.  The pushforward $\pi_{S*}(E|_{\ell \times S}\otimes \mathcal{O}(j))$ is a vector bundle on $S$ and we have chosen a non-zero point in the fiber over $s$.  Choose $T'$ open in $S$ and containing $s$ and an extension of the above section to an element of   \[H^{0}(T',(\pi_{S*}(E|_{\ell \times S}\otimes \mathcal{O}(j)))|_{T'})= H^{0}(\ell \times T', E|_{ \ell \times T'} \otimes \mathcal{O}(j))
\] such that this chosen global extension  does not vanish on $\ell \times T'$ and hence does not vanish on $Z_{k}^{(n)} \times T'$, and passes through our chosen element of the fiber over $s$.  This gives us an injective map of constant rank leading to a short exact sequence on $Z_{k}^{(n)} \times T'$
\[0 \to \mathcal{O}(-j) \to E_{Z_{k}^{(n)} \times T'} \to L \to 0
\]
where $L$ is a line bundle on $Z_{k}^{(n)} \times T'$ isomorphic to $\mathcal{O}(j)$ over every geometric point of $T'$.  By the see-saw principle there is a $T$ open in $T'$ and containing $S$ such that the restriction of $L$ to $Z_{k}^{(n)} \times T'$ is isomorphic to $\mathcal{O}(j)$.  Therefore the restriction of the above short exact sequence to $Z_{k}^{(n)} \times T$ gives the desired result.
\end{proof}

\begin{remark}
An alternate approach to the above two lemmas is to start with any vector bundle which has nowhere zero map of $\mathcal{O}(-j)$ to $E$ over $\ell \times T$ for some affine scheme $T$ and use the fact that $H^{1}(\ell \times T, I^{m}_{\ell \times T})=0$ for $m>0$ to extend this map order by order to a map over $Z_{k}^{(n)} \times T$ which must be nowhere zero.
\end{remark}
 \begin{lemma}\label{format}
 Let $T$ be an affine scheme and $E$ an algebraic extension of $\mathcal{O}_{Z^{(n)}_k \times T}$ modules
$$0 \rightarrow {\mathcal O}(-j) \rightarrow E
\rightarrow {\mathcal O}(j) \rightarrow 0,$$  over $Z^{(n)}_k \times T$  which splits over $\ell \times T$
for $j \geq 0$ then, in the chosen coordinates
  $E$  can be described by a transition  matrix
of the form
$$\left(\begin{matrix}  z^j & p \cr 0 &  z^{-j} \cr \end{matrix}\right)$$
on $(U^{(n)} \cap V^{(n)}) \times T,$  where

\begin{equation}
\label{pform}
p = \sum_{i = 1}^{min(\lfloor (2j-2)/k\rfloor,n-1)} \sum_{l = ki-j+1}^{j-1}p_{i,l}z^lu^i.
\end{equation}
and $p_{i.l} \in \mathcal{O}(T)$.
\end{lemma}

\begin{proof}
This follows immediately from a \v{C}ech cohomology calculation of 
\[\Ext^{1}_{Z_{k}^{(n)}}(\mathcal{O}(j),\mathcal{O}(-j)) = \frac{\mathbb{C}[z,z^{-1},u]/(u^{n})}{z^{-j}\mathbb{C}[z^{-1},z^{k}u]/((z^{k}u)^{n})+z^{j}\mathbb{C}[z,u]/(u^{n})}
\]
(performed in Theorem 3.3 of \cite{CA1}) along with equation (\ref{LerayPres}).
\end{proof}

\begin{corollary} As a consequence of the above two Lemmas \ref{format} and \ref{reps}, we see that any rank $2$ vector bundle on $Z_{k}^{(n)} \times T$ (or $\widehat{Z_{k}} \times T$) takes a special form locally on $T$ and in this form it is clearly the restriction (completion) of a vector bundle on $Z_{k}$. The theorem on formal functions implies then that
\[\widehat{\Ext^{i}_{Z_{k} \times T}(V,W)} \cong \Ext^{i}_{\widehat{Z_{k}} \times T}(V,W)
.\]
\end{corollary}

\ \hfill $\Box$

\begin{notation}

Let
\[N_{j,k}^{(n)}= \{ (i,l)| ki-j+1 \leq l \leq j-1\phantom{x}\text{and}\phantom{x} 1 \leq i \leq min(\lfloor (2j-2)/k\rfloor,n-1)\}.\]
Consider the algebraic variety over $\mathbb{C}$
\begin{equation}
W_{j,k}^{(n)} =  Spec \left(\mathbb{C}[\phantom{x}p_{i,l}\phantom{x}| \phantom{x} (i,l) \in N_{j,k}^{(n)}] \right)
.\end{equation}
For any fixed $j,k$ it remains finite dimensional even for $n=\infty$.  If we pass to the $\mathbb{C}$ points then 
we get
\[W^{(n)}_{j,k}(\mathbb{C}) =  \{p \in \Ext_{Z_{k}^{(n)}}^1({\mathcal O}(j) , \mathcal O(-j))\phantom{x}|\phantom{x}
p \vert_\ell = 0\}.\]

Let
\begin{equation}\label{Rjeqn}
R^{(n)}_{j,k} = \bigoplus_{i=1}^{\lfloor (2j-2)/k\rfloor}  \bigoplus_{l = ki-j+1}^{j-1} \mathbb{C} z^{l} u^{i} \subset \mathcal{O}(U^{(n)} \cap V^{(n)}).
\end{equation}
of course $R^{(n)}_{j,k}$ is the set of $\mathbb{C}$ points of  $W_{j,k}^{(n)}$ but we distinguish them because of the different notions of automorphisms of $R^{(n)}_{j,k}$ and $W_{j,k}^{(n)}$.
\end{notation}

\begin{remark}
Note that in our chosen form of transition matrix  from the above theorem we have explicitly chosen
$p\in R_{j,k}^{(n)}$.
\end{remark}
\begin{definition}\label{bigbundle}
Consider the open cover
$\{ U^{(n)} \times W_{j,k}^{(n)} , V^{(n)} \times W_{j,k}^{(n)} \}$ of
$Z^{(n)}_k \times W_{j,k}^{(n)}$.  We define $\mathbb{E}$, sometimes called the big bundle, to be the bundle 
$$\begin{array}{c}
 \mathbb{E} \\
\downarrow  \\
Z^{(n)}_k\times W_{j,k}^{(n)}
\end{array}$$
on $Z_{k} \times W_{j,k}^{(n)}$ defined by
transition matrix
$$\left(\begin{matrix}
 z^j & p \cr 0 &  z^{-j} \cr \end{matrix}\right) \in H^{0}((U^{(n)} \cap V^{(n)})
\times W^{(n)}_{j,k},\mathcal{A}ut(\mathcal{O}^{\oplus 2})).$$
Let $T$ be an affine scheme and $p$ a morphism from $T$ to $W_{j,k}^{(n)}$.  We denote by $E_{p}$ the bundle (also described in Lemma \ref{format}) given by the  pullback $(\text{id}_{Z_{k}^{(n)} } , p)^{*} \mathbb{E}$ of $\mathbb{E}$ via the map
\[Z_{k}^{(n)} \times T \stackrel{(\text{id}_{Z_{k}^{(n)} } , p)}\to  Z_{k}^{(n)} \times W_{j,k}^{(n)}.
\]\end{definition}
\begin{lemma}\label{proj} \cite[thm. 4.9]{BGK}
On the first formal neighborhood $Z_{k}^{(1)}$,
two  bundles $E$ and $E'$ with
transition matrices
$$\left(\begin{matrix} z^j & p_1   \cr 0 & z^{-j} \cr\end{matrix}\right)\,\,\,
and \,\,\,
\left(\begin{matrix} z^j & p^\prime_1   \cr 0 & z^{-j} \cr\end{matrix}\right)$$
 respectively  are isomorphic if and only if
$p^\prime_1 = \lambda p_1$ for some
 $\lambda \in  {\mathbb C}^{\times}$.
\end{lemma}

\begin{remark} It follows from this lemma that the coarse moduli space of bundles on $Z_{k}^{(1)}$
coming from non-trivial extensions of $\mathcal{O}(j)$ by $\mathcal{O}(-j)$ is isomorphic to
$ \mathbb P^{2j-k-2}\text{.}$
\end{remark}

\begin{example} On higher infinitesimal neighborhoods we need to consider
far more relations among extension classes then just projectivisation
to obtain the moduli of bundles. The simplest of such examples
occurs in the  case when $k=1$ and
 $j = 2$, so that our
extension classes  have the form
\[p =( p_{1,0} + p_{1,1} z)u + p_{2,1} zu^2.
\]
The set of equivalence classes of vector bundles is then ${\mathbb C}^3 / \sim$ where the equivalence relation is generated by
\begin{itemize}
\item[]  ${(p_{1,0},p_{1,1},p_{2,1}) \sim  (\lambda p_{1,0},
 \lambda p_{1,1}, \lambda p'_{2,1})}\,\,\,
if \, \,\,(p_{1,0},p_{1,1}) \neq (0,0),\,\,\lambda \neq 0,$
\item[]  ${(0,0,p_{2,1}) \sim  ( 0,0, \lambda p_{2,1})},
\,\,\lambda \neq 0.$
\end{itemize}
Note that $p_{2,1}'$ is does not depend on $p$, and that the
quotient topology makes the entire space the only open neighborhood
of the split bundle, which is the image of the origin in $\mathbb C^3$.
\end{example}

\subsection{Stacks of vector bundles}

We now define the stack  of bundles $\mathfrak{M}_{j}(Z_{k}^{(n)})$, the main object we seek to understand in this article.

\begin{definition}
\[\mathfrak{M}_{j}(Z_{k}^{(n)})\colon \text{Schemes} \to \text{Groupoids}
\]
given by
\[T \mapsto \Hom(T,\mathfrak{M}_{j}(Z_{k}^{(n)}))
\]
where
\begin{equation}\begin{split}\text{ob}(\Hom(T,\mathfrak{M}_{j}(Z_{k}^{(n)})) = 
\{ & \text{rank $2$ vector bundles on} \ \ Z^{(n)}_k \times T \  \ \text{which have} \\
&\text{splitting type j and first Chern class 0 for every}
\\
& \text{restriction to} \ \ \ Z_{k}^{(n)} \times \{t\},  t \in T(\mathbb{C}) \}
\end{split}
\end{equation}
and
\[\text{mor}(\Hom(T,\mathfrak{M}_{j}(Z_{k}^{(n)}))(V_{1},V_{2}) = \Isom(V_{1},V_{2}).
\]
\end{definition}
This is a stack \cite{Lau} with respect to the faithfully flat topology on schemes ($\mathbb{C}$-algebras).
Notice that there is automatically a universal bundle $\mathcal{E}$ over $Z^{(n)}_{k} \times \mathfrak{M}_{j}(Z_{k}^{(n)})$.  We can similarly define the stack  $\mathfrak{M}_{j}(\widehat{Z_{k}})$.  The stacks $\mathfrak{M}_{j}(\widetilde{Z_{k}})=\mathfrak{M}_{j}(Z_{k}^{(\infty)})$ and $\mathfrak{M}_{j}(\widehat{Z_{k}})$ are isomorphic and we will not distinguish between them.  We similarly have the stacks $\mathfrak{M}(Z_{k}^{(n)})$ of bundles where we drop the condition on splitting type.

There is an inverse (or projective) system of stacks of finite type over $\mathbb{C}$:
\begin{equation}\label{invlim}\cdots \to  \mathfrak{M}_{j}(Z_{k}^{(3)}) \to \mathfrak{M}_{j}(Z_{k}^{(2)}) \to  \mathfrak{M}_{j}(Z_{k}^{(1)}) \to \mathfrak{M}_{j}(Z_{k}^{(0)}) = \mathfrak{M}_{j}(\mathbb{P}^{1})
\end{equation}
whose inverse limit in the category of algebraic stacks is $\mathfrak{M}_{j}(\widetilde{Z_{k}})$.  Alternatively we can consider the inverse system $\mathfrak{M}_{j}(Z_{k}^{(\bullet)})$ to be an pro-stack of pro-finite type.  This type of approximation is studied in \cite{Ry}.  It seems difficult to compute invariants of the stacks $\mathfrak{M}_{j}(Z_{k}^{(n)})$ using only the definition above so we will find a more explicit description below. 

\subsection{The structure of vector bundle 
isomorphisms}  
Consider the bundles $E_{p}$ defined in Definition \ref{bigbundle}.  There is an exact sequence 
\begin{equation}\label{HomSequence}
\begin{split}
0 \to \Hom(E_{p},E_{p'}) \to & \End(\mathcal{O}(-j) \oplus \mathcal{O}(j)) \stackrel{}\longrightarrow  \\
& \Ext^{1}(\mathcal{O}(-j) \oplus \mathcal{O}(j), \mathcal{O}(-j) \oplus \mathcal{O}(j)) \to \Ext^{1}(E_{p},E_{p'}) \to 0 \text{.}
\end{split}
\end{equation}
We now explain the structure of isomorphisms between families of bundles coming from extensions by constructing an explicit splitting for the first non-trivial map in this sequence.
If  the bundles $E_{p}$ and $E_{p'}$ on $Z^{(n)}_k \times T$, given by maps
\[p,p': T \to R^{(n)}_{j,k}\] are isomorphic
(see equation (\ref{Rjeqn})) then necessarily they have the same splitting type, and in such case
we can represent them by
transition matrices on \[(U^{(n)} \cap V^{(n)})\times T\] by
$\left(\begin{matrix} z^j & p  \cr 0 & z^{-j} \cr\end{matrix}\right)$ and
$\left(\begin{matrix} z^j & p'  \cr 0 & z^{-j} \cr\end{matrix}\right) $
respectively.  An isomorphism between $E_{p}$ and
$E_{p'}$ is given by a pair of invertible matrices \[A =\left(\begin{matrix} a_{U} & b_{U}  \cr c_{U} & d_{U} \cr\end{matrix}\right)\] regular on $U^{(n)} \times T$ and
\[B =\left(\begin{matrix} a_{V} & b_{V}  \cr c_{V} & d_{V} \cr\end{matrix}\right)\]  regular on $V^{(n)} \times T$, such that:
\begin{equation}\label{ABeqnWithMinus}
B \left(\begin{matrix} z^j & p  \cr 0 & z^{-j} \cr\end{matrix}\right)  =
\left(\begin{matrix} z^j & p'  \cr 0 & z^{-j} \cr\end{matrix}\right) A,
\end{equation}
or equivalently
\begin{equation}\label{Beqn}
B =
\left(\begin{matrix} z^j & p'  \cr 0 & z^{-j} \cr\end{matrix}\right) A
  \left(\begin{matrix} z^{-j} & -p  \cr 0 & z^j \cr\end{matrix}\right) = \left(\begin{matrix} a_{U} + z^{-j}p'c_{U} & 
z^{2j}b_{U} +z^j(p'd_{U}  -a_{U}
p) - pp'c_{U}  \cr
z^{-2j}c_{U} & d_{U} - z^{-j}pc_{U} \cr\end{matrix}\right)
\text{.} 
\end{equation}

\begin{definition}
We use the notation $Y^+$ to denote the terms in $Y \in \mathcal{O}((U^{(n)} \cap V^{(n)}) \times T)$ that are not regular on $V^{(n)} \times T$ and  $Y^{+,\geq 2j}$ denotes the terms in $Y$ that are not regular on $V^{(n)} \times T$ and have power of $z$ greater than or equal to $2j$.
\end{definition}

\begin{lemma}\label{iso}
Suppose that $j>0$.  Then a general isomorphism between $E_{p}$ and $E_{p'}$ on $Z^{(n)}_k \times T$
has the form 
\[(A,B) = (M_{U},M_{V}) + (\Phi_{U}(M),\Phi_{V}(M))
\]
\[M_{U} = \left(\begin{matrix}  \underline{a} & \underline{b}_{U} \cr \underline{c}_{U} & \underline{d}
\cr\end{matrix}\right) 
\]
and 
\[\Phi_{U}(M) =  \left(\begin{matrix}  - (z^{-j}p'\underline{c}_{U})^{+} & -z^{-2j}\left(z^j(p'\underline{d}   -\ua p) - pp'\underline{c}_{U}\right)^{+,\geq 2j} \cr 0  &  (z^{-j}p\underline{c}_{U})^{+} 
\cr\end{matrix}\right)
\]

where \[M \in  \Aut_{Z_{k}^{(n)} \times T}(\mathcal{O}(j)\oplus \mathcal{O}(-j)).
\]
satisfying
\begin{equation}\label{orthis}
[p'\underline{d}   -\ua p - z^{-j} pp'\uc_{U}] =0 \in \Ext^{1}(\mathcal{O}(j),\mathcal{O}(-j)).\end{equation}
is uniquely determined by the isomorphism.

\end{lemma}

\begin{proof}
First suppose that such an isomorphism exists, between $E_{p}$ and $E_{p'}$.  Then we have
\begin{equation}\label{addversion}
\left(\begin{matrix} z^j & p'  \cr 0 & z^{-j} \cr\end{matrix}\right) A - B \left(\begin{matrix} z^j & p  \cr 0 & z^{-j} \cr\end{matrix}\right)  = 0.
\end{equation}
The left hand side comes out to be
\begin{equation}\label{Zeromatrix}
\left(\begin{matrix} p'c_{U}+(a_{U}-a_{V})z^{j} & d_{U}p'-a_{V}p+z^{j}b_{U}-z^{-j}b_{V}  \cr c_{U} z^{-j}-c_{V} z^{j} & z^{-j}(d_{U}-d_{V})-c_{V}p \cr\end{matrix}\right) .
\end{equation}
The lower left corner of (\ref{Zeromatrix}) implies first of all that $c$ must be a section $\uc$ of $\mathcal{O}(2j)$.  
We need to arrange for the vanishing of all terms in (\ref{Zeromatrix}).
Therefore, we need to solve the equations:
$$\begin{array} {l}
a_{U}-a_{V} = -z^{-j}p'\uc_{U}\cr
z^{j}b_{U}-z^{-j}b_{V}  =-d_{U}p'+a_{V}p  \cr
d_{U}-d_{V}= z^{j}\uc_{V}p\text{.} \end{array}$$
Becasue $H^{1}(Z_{k}^{(n)} \times T,\mathcal{O})$ vanishes, the first and third equations have solutions which are unique up to global functions.
Let 
\[a_{U}= \ua - (z^{-j}p'\uc_{U})^{+} 
\] and 
\[d_{U} =\ud +(z^{j}\uc_{V}p)^{+} .\]  These solved the first and third equation.  If we substitute into the second equation, it reads
\begin{equation}\label{itreads}\begin{split}z^{j}b_{U}-z^{-j}b_{V}  &=-(z^{j}\uc_{V}p)^{+}p'+(-(z^{-j}p' \uc_{U})^{+}+z^{-j}p' \uc_{U})p -\ud p'+\ua p  \\
& = -\ud p'+\ua p  +z^{-j}pp'\uc_{U}
\text{.}
\end{split}
\end{equation}
This implies that
\[ 
[p'\underline{d}   -\ua p - z^{-j} pp'\uc_{U}] =0 \in \Ext^{1}(\mathcal{O}(j),\mathcal{O}(-j)).
\]
Conversely, suppose that these conditions are satisfied by some $\underline{a}$, $\underline{d}$, $\uc$, $p$, and $p'$, let us record the general form of an element of $\Isom_{Z^{(n)}_k \times T}(E_{p},E_{p'})$.  It remains only to determine the expression for $b_{U}$.
By the assumptions we already know that
\[\left(z^j(p'\underline{d}   -\ua p) - pp'\uc_{U}\right)^{+,<2j}
\]
is regular on $V^{(n)} \times T $.
Hence
\[b_{U} = \ub_{U}-z^{-2j}\left(z^j(p'\underline{d}   -\ua p) - pp'\uc_{U}\right)^{+,\geq 2j}.
\]
Finally, since $u$ divides $p$ and $p'$, we know that $A$ is invertible if and only if $M_{U}$ is.
\end{proof}

\begin{remark}
We conclude that the expression of the element $(A,B)$ of $\Hom(E_{p},E_{p'})$ under the decomposition (\ref{ssdecomp}) \[\Hom(E_{p},E_{p'}) =\Hom(\mathcal{O}(j),\mathcal{O}(-j)) \oplus \phi(\text{ker}(d_{1}^{1,-1})) \oplus \psi(\text{ker}(d_{2}^{0,0}))
\] from the appendix is satisfied if we take $\underline{b} \in \Hom(\mathcal{O}(j),\mathcal{O}(-j)),$
\[\psi_{U}(c) = \left(\begin{matrix}  - (z^{-j}p'\uc_{U})^{+}  & z^{-2j}\left(  pp'\uc_{U}\right)^{+,\geq 2j} \cr \uc _{U}& (z^{-j}p\uc_{U})^{+} \end{matrix}\right)
\]
and
\[\phi_{U}(\underline{a},\underline{d}) =  \left(\begin{matrix}  \underline{a}  &-z^{-2j}\left(z^j(p'\underline{d}   -\ua p) \right)^{+,\geq 2j} \cr 0 & \underline{d} \end{matrix}\right).
\]
\end{remark}

\subsection{Bundle isomorphism viewed as an equivalence relation}\label{Mob}
Although we have worked out the structure of the space of isomorphisms between two given bundles, this does not yet give a criterion for when two bundles are isomorphic nor does it provide any understanding of the equivalence relation on $W_{j,k}^{(n)}$ given by isomorphisms of vector bundles.  We show that there are algebraic groups $G_{j,k}^{(n)}$ acting on $W_{j,k}^{(n)}$ so that the orbits of this action are identified with the equivalence classes.   This action (\ref{nobulletmobius}) takes on the familiar form of a M\"obius transformation.  Lange studied in \cite{LA} (see also Dr\'ezet \cite{Dr}) the question of universal bundles over the projectivized space of extensions.  In a specific example we study here a more difficult problem, the difference being that we do not remove the origin and we consider all vector bundle isomorphisms, not just those that correspond to scaling the extension.  First we need to define the structure of a scheme on the sets $\Aut_{Z_{k}^{(n)}}(\mathcal{O}(j) \oplus \mathcal{O}(-j))$ for $n$ finite.

\begin{definition}

Consider the functors from schemes to sets given by
\[T \mapsto \Aut_{Z_{k}^{(n)} \times T}(\mathcal{O}(j) \oplus \mathcal{O}(-j)).
\]
These functors are $\mathbb{C}$-groups (sheaves of groups in the faithfully flat topology on schemes) and are easily seen to be representable by reduced schemes.  They are affine, being defined inside the finite dimensional affine space
\[\mathbb{E}\text{nd}_{Z_{k}^{(n)}}(\mathcal{O}(j) \oplus \mathcal{O}(-j))
\] defined with coordinates as in \ref{spectrum} as the complement of the pre-image of $0$ by the morphism 
\[det_{0}: \mathbb{E}\text{nd}_{Z_{k}^{(n)}}(\mathcal{O}(j) \oplus \mathcal{O}(-j)) \to \mathcal{O}(Z_{k}^{(n)}) \to Spec(\mathbb{C}[s]).
\]
sending $s$ to the restriction of the determinant to $\ell$.  When we pass to $\mathbb{C}$ points we get the standard determinant followed by restriction to $\ell$
\[det_{0}: \End_{Z_{k}^{(n)}}(\mathcal{O}(j) \oplus \mathcal{O}(-j)) \to \mathcal{O}(Z_{k}^{(n)}) \to \mathcal{O}(Z_{k}^{(0)})=\mathbb{C}.
\]
We denote these finite dimensional algebraic groups by $G_{j,k}^{(n)} $.
These form a directed system of $\mathbb{C}$-spaces (sheaf of sets for the faithfully flat topology on the category of
$\mathbb{C}$-algebras) and their direct limit as a $\mathbb{C}$-space (see \cite{BL1} for this yoga) is representable by an infinite dimensional algebraic variety,  \[\widetilde{G_{j,k}}= G_{j,k}^{(\infty)}\] which has $\Aut_{Z_{k}^{(\infty)}}(\mathcal{O}(j) \oplus \mathcal{O}(-j))$ as its underlying set of $\mathbb{C}$ points.  In fact, it is an infinite-dimensional algebraic group.
 $\Aut_{Z_{k}^{(\infty)}}(\mathcal{O}(j) \oplus \mathcal{O}(-j))$.   The sequence $G_{j,k}^{(\bullet)}$

\begin{equation}\label{groupinvlim}\cdots \to  G_{j,k}^{(3)} \to G_{j,k}^{(2)}  \to G_{j,k}^{(1)} \to G_{j,k}^{(0)}  = \Aut_{\mathbb{P}^{1}}(\mathcal{O}(j)\oplus \mathcal{O}(-j))
\end{equation}
is an pro-finite-type pro-scheme.  We often write elements of $\Hom(T,G_{j,k}^{(n)})$ as matracies 

\end{definition}
Because of equation (\ref{Beqn}) there must exist some function $h$, regular on $V \times T$ such that
\[p' =  \frac{pz^{j} \underline{a} + h}{z^{j}\underline{d} - p\uc_{U}} = \frac{z^{j} \underline{a} (p+ \frac{h}{z^{j} \underline{a}})}{z^{j}\underline{d} - p\underline{c}_{U}}\text{.}
\]
So we only need to know if $h$ can affect $p$ in cohomology.  $h$ is a linear combination of terms $z^l u^i$ with $l+j \leq ki$.  Therefore $\frac{h}{z^{j}\underline{a}}$ has terms of the form $z^{l} u^i$ with $l+2j \leq ki$. However, the terms in the  correct range have $1 \leq i \leq \lfloor \frac{2j-1}{k} \rfloor$  and $ki-j+1 \leq l \leq j-1$.  So if $\frac{h}{z^{j}\underline{a}}$ had a term in the correct range we would have
\[l+j+1 \leq  ki-j+1 \leq l
\]
which is a contradiction.
Thus none of the terms appearing in $\frac{h}{z^{j}\underline{a}}$ are in the correct range.   So we can drop $h$ all together and conclude that the action of $\Hom(T, G_{j,k}^{(n)})$ on $\Hom(T, W_{j,k}^{(n)})$ is
\[p'=\frac{\ua p}{\underline{d}-z^{-j}p\underline{c}_{U}} = \left(\frac{\underline{a}}{\underline{d}}\right)p\sum_{t=0}^{\infty}\left(\frac{z^{-j}p\underline{c}_{U}}{\underline{d}}\right)^{t} \]
Consider the following direct sum
 decomposition of the vector space of functions
\[\mathcal{O}_{Z_{k}^{(n)}}(U^{(n)} \cap V^{(n)}) = \mathcal{O}_{Z_{k}^{(n)}}(U^{(n)} \cap V^{(n)})^{\succ} \oplus \mathcal{O}_{Z_{k}^{(n)}}(U^{(n)} \cap V^{(n)})_{good} \oplus \mathcal{O}_{Z_{k}^{(n)}}(U^{(n)} \cap V^{(n)})^{\prec}
\]
where the sector named ``good'' corresponds to the terms appearing in Theorem \ref{reps}, and also
\[z^{j}\mathcal{O}_{Z_{k}^{(n)}}(U^{(n)} \cap V^{(n)})^{\prec} \subset \mathcal{O}(V^{(n)})\]
and
 \[z^{-j}\mathcal{O}_{Z_{k}^{(n)}}(U^{(n)} \cap V^{(n)})^{\succ} \subset \mathcal{O}(U^{(n)})\]
\[q - q_{good} = q^{\succ} +q^{\prec}.
\]
As in equation (\ref{ABeqnWithMinus}) we write elements of 
\[\Hom(T,G_{j,k}^{(n)})\subset H^{0}(Z_{k}^{(n)} \times T, \mathcal{O}^{\oplus 2} \oplus \mathcal{O}(2j) \oplus \mathcal{O}(-2j))
\] in the form
\begin{equation}
g=\left(\begin{matrix} \underline{a} & \underline{b}  \cr \underline{c} & \underline{d} \cr\end{matrix}\right) .
\end{equation}
with $\underline{b}= (\ub_{U},\ub_{V})$ and $\ub_{U}$ holomorphic on $U^{(n)} \times T$, etc.
First of all notice that the groups $\Hom(T,G_{j,k}^{(n)})$ acts on the functions $p$ on  $U^{(n)} \cap V^{(n)} \times T$  which vanish on the zero section by the formula
\begin{equation}
\label{nobulletmobius}
gp=\frac{\underline{a}p- z^{j}\underline{b}_{U}}{\underline{d}-z^{-j}p\underline{c}_{U}} .
\end{equation}
A special case  (where $\underline{b}$ and $\underline{c}$ are taken to be zero) of this action was observed for general varieties and bundles in \cite{Dr}.  For $n$ finite, such functions vanishing on $\ell$ belong to  $u\mathbb{C}[z,z^{-1}][u]/(u^{n})$, in the case $n=\infty$ such functions belong to $u\mathbb{C}[z,z^{-1}][[u]]$.  
The action $p \mapsto gp$ does not preserve the finite dimensional space $R_{j,k}^{(n)}$ which was written in (\ref{Rjeqn}).

\begin{definition}\label{actdef}
So we need to correct this action to a morphism

\[
G_{j,k}^{(n)} \times R^{(n)}_{j,k} \to R^{(n)}_{j,k}
\]
\begin{equation}\label{correctedMob}
\begin{split}
(g,p) \mapsto g \bullet p &=  \frac{\ua p- z^{j}\ub_{U}}{\ud-z^{-j}p\uc_{U}} - \left(\frac{\ua p- z^{j}\ub_{U}}{\ud-z^{-j}p\uc_{U}} \right)^{\succ} - \left( \frac{\ua p- z^{j}\ub_{U}}{\ud-z^{-j}p\uc_{U}} \right)^{\prec}.
\\
&=\left( \frac{\ua p- z^{j}\ub_{U}}{\ud-z^{-j}p\uc_{U}} \right)_{good}
\end{split}
\end{equation}
which will become one of the structure maps of a groupoid (see equation (\ref{tdef})).  It is not the action of a group.
\end{definition}

Consider 
\begin{equation}\label{Adef}
A_{g}(p) = \left(\begin{matrix} \ua - (z^{-j}p\uc_{U})^{+}  &
\ub_{U}  -z^{-2j}\left(z^j\left((g\bullet p)\ud  -\ua p\right) - p(g\bullet p)\uc_{U}\right)^{+,\geq 2j} \cr
\uc_{U} & \ud + (z^{-j}\uc_{U}(g\bullet p))^{+} \cr\end{matrix}\right)
\end{equation}
and

\begin{equation}\label{Bdef}
B_{g}(p) =\left(\begin{matrix} \ua +(z^{-j}p\uc_{U})^{+} &
\ub_{V}+ \left(z^j\left((g\bullet p)\ud  -\ua p\right) - p(g\bullet p)\uc_{U}\right)^{+,< 2j}  \cr
\uc_{V} & \ud-  (z^{-j}\uc_{U}(g\bullet p))^{+}  \cr\end{matrix}\right).
\end{equation}
They are holomorphic over $U^{(n)} \times T$ and $V^{(n)} \times T$.
They satisfy
\begin{equation}\label{provide}B_{g}(p) \left(\begin{matrix} z^j & p  \cr 0 & z^{-j} \cr\end{matrix}\right)  =
\left(\begin{matrix} z^j & g\bullet p  \cr 0 & z^{-j} \cr\end{matrix}\right) A_{g}(p)
\end{equation}
and so the pair $(A_{g}(p),B_{g}(p))$ provides an isomorphism between $E_{p}$ and $E_{g\bullet p}$.  We have shown the following Lemma.
\begin{lemma}
There is a morphism 
\[
G_{j,k}^{(n)} \times R^{(n)}_{j,k} \to R^{(n)}_{j,k}
\]
\begin{equation}
(g,p) \mapsto g \bullet p
\end{equation}
such that for two bundles $E_{p}$ and $E_{p'}$ of constant splitting type $j$, 
\begin{equation}\begin{split}\Isom_{Z^{(n)}_k \times T}(E_{p}, E_{p'}) & =\{g \in \Hom(T,
G_{j,k}^{(n)}) \phantom{x}|\phantom{x} g \bullet p = p' \} \\
& = \{g \in \Hom(T,
G_{j,k}^{(n)}) \phantom{x}|\phantom{x} \text{\ref{orthis} is satisfied} \}.
\end{split}
\end{equation}
\end{lemma}

\ \hfill $\Box$

\begin{definition}\label{dotp}
Consider the isomorphism \[(A_{g_{1}}( g_{2} \bullet p) A_{g_{2}}(p), B_{g_{1}}(g_{2} \bullet p) B_{g_{2}}(p))\] between $E_{p}$ and $E_{g_{1} \bullet (g_{2} \bullet p)}$.  By lemma \ref{iso}, there is a unique element \[g_{1} \bullet_{p} g_{2} \in G_{j,k}^{(n)}(\mathbb{C})\]  such that this isomorphism equals $(A_{g_{1}\bullet_{p} g_{2}}, B_{g_{1}\bullet_{p} g_{2}})$.   Similarly, the isomorphism $(A_{g}(p)^{-1},B_{g}(p)^{-1})$ between $E_{g\bullet p}$ and $E_{p}$ corresponds to a unique 
\begin{equation}\label{pinv}g^{(-1)_{p}}\in G_{j,k}^{(n)}(\mathbb{C}).\end{equation}   From here it is clear (since both $A_{e_{G_{j,k}^{(n)}}}(p)$ and $B_{e_{G_{j,k}^{(n)}}}(p)$ are the identity matrix) that 
\begin{equation}\label{annoying}g \bullet_{p} g^{(-1)_{p}} = e_{G_{j,k}^{(n)}}=g^{(-1)_{p}} \bullet_{p} g.
\end{equation}
The elements $g_{1} \bullet_{p} g_{2}$ vary  algebraically with $g_{1}$ and $g_{2}$ and give a morphism of schemes 
\[G_{j,k}^{(n)} \times G_{j,k}^{(n)}  \times W_{j,k}^{(n)} \to G_{j,k}^{(n)}
\]
\[(g_{1},g_{2},p) \mapsto g_{1}\bullet_{p} g_{2} \text{.}
\]
The restriction to $p=0$ in $W_{j,k}^{(n)}$ gives us back the standard multiplication but in general this structure does depend on $p$.
\end{definition}

Therefore by definition we have 
\begin{equation} \label{defGrpStr}
B_{g_{1}}(g_{2} \bullet p) B_{g_{2}}(p) = B_{g_{1} \bullet_{p} g_{2}}(p).
\end{equation}
(and also $A_{g_{1}}( g_{2} \bullet p) A_{g_{2}}(p) = A_{g_{1} \bullet_{p} g_{2}}(p)$).
An immediate consequence of this together with (\ref{provide}) is
\begin{equation}\label{really}g_{1}\bullet (g_{2} \bullet p) = (g_{1} \bullet_{p} g_{2}) \bullet p \text{,}
\end{equation}
and we also have
\begin{equation}
\begin{split}
B_{(g_{1} \bullet_{(g_{3} \bullet {p})} g_{2}) \bullet_{p} g_{3}}(p) & = B_{g_{1} \bullet_{(g_{3} \bullet {p})} g_{2}}( g_{3} \bullet p)B_{ g_{3}}(p) =
B_{g_{1}}(g_{2} \bullet (g_{3} \bullet p))B_{g_{2}}(g_{3} \bullet p) B_{g_{3}}(p) \\ & =B_{g_{1}}(g_{2} \bullet (g_{3} \bullet p)) B_{g_{2} \bullet_{p} g_{3}}(p) = B_{g_{1} \bullet_{p} (g_{2} \bullet_{p} g_{3})}(p)
\end{split}
\end{equation}
and similarly for  $A_{g}(p)$.  Because every isomorphism $(A,B)$ which takes one of our chosen  transitions matrices corresponding to a bundle $E_{p}$ to another transition matrix of the same form corresponds (\ref{HomSequence}) to a unique $g \in \Hom(T,G_{j,k}^{(n)})$ we  conclude that 
\begin{equation}\label{concludethat}
(g_{1} \bullet_{(g_{3} \bullet {p})} g_{2}) \bullet_{p} g_{3} = g_{1} \bullet_{p} (g_{2} \bullet_{p} g_{3}).
\end{equation}
This will be used to verify the associativity of the groupoid structure.  A direct inspection of (\ref{correctedMob}), (\ref{Adef}) and (\ref{Bdef}) shows that identity matrix $e_{G_{j,k}^{(n)}}$ satisfies 
\begin{equation}\label{idenprop}
e_{G_{j,k}^{(n)}} \bullet p = p  
\end{equation}
for any $p$ and corresponds to the identity map from $E_{p}$ to itself.   Therefore we of course have 
\begin{equation}\label{ofcourse}
e_{G_{j,k}^{(n)}} \bullet_{p} g = g = g \bullet_{p} e_{G_{j,k}^{(n)}} 
\end{equation}
for any $p$.

\section{An explicit groupoid in schemes}\label{DefOfGrpd}
In this section we describe an explicit groupoid in schemes and show that its associated stack is isomorphic to the stack of rank $2$ vector bundles of splitting type $j$ and first Chern class $0$ on $Z_{k}^{(n)}$.
\subsection{Review of groupoids in schemes and their sheaf theory}\label{RevGrp}
We begin with a review of the definition of a groupoid in schemes and the notion of a sheaf on a groupoid in schemes.
Recall that a groupoid
\[\mathcal{G} = (A,R,s,t,m,e,\iota)
\] in schemes consists of schemes $A$ (the atlas) and $R$ (the relations), morphisms $s,t,m,e, \iota$
\begin{equation}\label{GroupoidForGerbe}
\xymatrix{**[l] R  \ar@/_2pc/[r]_-{s} \ar@/^2pc/[r]^-{t} & **[r] \ar@/^0pc/[l]_-{e}  A}
\end{equation}

\[\xymatrix{{R}_{t} \times_{A} {{{}_{s}}R} \ar@/^0pc/[r]^-{m} & R}
\]
and
\[\xymatrix{R \ar@/^0pc/[r]^-{\iota} &R  }
\]
which satisfy some conditions which we write below.
Let $p_{1}, p_{2}$ be the first and second projections
\[R_{t} \times_{A} {}_{s}R \stackrel{p_{1},p_{2}} \longrightarrow R
\]
and let $\Delta$ be the diagonal
\[R_{t} \times_{A} {}_{s}R \stackrel{\Delta} \longleftarrow R.
\]
The morphisms then must satisfy
\begin{equation}
\label{cond1}
m \circ (m, \text{id}_{R}) = m \circ (\text{id}_{R},m)\end{equation} on all composable elements of $R \times R \times R$,
\begin{equation}\label{cond2}
t \circ m = t \circ p_{2}, \phantom{xx} s \circ m = s \circ p_{1}
\end{equation}
on all composable elements of $R \times R$
\begin{equation}
\label{cond3}
m \circ (\iota, \text{id}_{R}) \circ \Delta = e \circ s, \phantom{xx} m \circ (\text{id}_{R},\iota) \circ \Delta = e \circ s
\end{equation} on $R$,
and also
\begin{equation}
\label{cond4}
m \circ (\text{id}_{R},e \circ t) \circ{\Delta}= \text{id}_{R}, \phantom{xx} m\circ(e \circ s, \text{id}_{R}) \circ{\Delta} = \text{id}_{R}
\end{equation}
on $R$.  Notice that for any scheme $S$ that by taking the set of morphisms of schemes from $S$ into $R$ and $A$ one gets a pair of sets and these naturally form a groupoid in sets using the obvious maps.    We denote this groupoid in sets by
\[\Hom(S,\mathcal{G})
.\]
A (coherent/locally free of rank $r$) sheaf of modules on the groupoid consists of a (coherent/locally free of rank $r$) sheaf $\mathcal{S}$  of $\mathcal{O}_{A}$ modules on $A$ together with an isomorphism $f$ of sheaves of $\mathcal{O}_{R}$ modules over $R$
\[f:s^{*} \mathcal{S} \to t^{*} \mathcal{S}
\]
which satisfies
\begin{equation}\label{isomcond1}p_{2}^{*}f \circ p_{1}^{*} f = m^{*}f
\end{equation}
and
\begin{equation}\label{isomcond2}
e^{*}f = \text{id}.
\end{equation}
To make sense of this equality, one must use the identities
\[s \circ p_1 = s \circ m,\phantom{xx}\text{and}\phantom{xx}  t \circ p_2 = t \circ m.
\]

\subsection{Stacks from groupoids} \label{Grpds2Stacks}

Let $\mathcal{G} = (A,R,s,t,m,e,\iota)
$ be a groupoid in schemes.

We associate to it a stack $[\mathcal{G}]$ defined as the stack on the fppf site associated to the prestack $\text{pre-}[\mathcal{G}]$ which associates to any test scheme $T$ the groupoid in sets
\[\text{pre-}[\mathcal{G}](T)=\Hom(T,\mathcal{G}).
\]
Notice that such a morphism consists of a map from maps from $T$ to $A$, and $T$ to $R$ which satisfy the obvious compatibilities.
\begin{remark}
In the case that $R=G \times A$ and the groupoid structure is just given by a group action of $G$ on $A$, we may denote the associated quotient stack by $[A/G]$, leaving the structure implicit.
\end{remark}

There is an equivalence \cite{Lau} of Abelian categories of coherent sheaves which takes vector bundles to vector bundles
\begin{equation}\label{CohEquiv}  \text{Coh}(\mathcal{G})  \stackrel{\cong}\longrightarrow   \text{Coh}([\mathcal{G}]).
\end{equation}
\begin{definition}
We denote by  $[\mathcal{S}]$ the sheaf on $[\mathcal{G}]$ corresponding to a sheaf $\mathcal{S}$ on $\mathcal{G}$ under the equivalence \ref{CohEquiv} given above.
\end{definition}
\subsection{Groupoid presentations for stacks of rank $2$  bundles}
We define a groupoid in schemes to be called $\mathcal{G}_{j,k}^{(n)}$.  The atlas of  $\mathcal{G}_{j,k}^{(n)}$ is  $W_{j,k}^{(n)}$ and the relations are $G_{j,k}^{(n)} \times W_{j,k}^{(n)}$.

The arrow $s$ is given by the projection
\[G_{j,k}^{(n)} \times W_{j,k}^{(n)}  \stackrel{s}\to W_{j,k}^{(n)} .
\]
defined by
\[(g , p) \mapsto p.
\]
The arrow $t$ is given by the map
\begin{equation}\label{tdef}G_{j,k}^{(n)} \times W_{j,k}^{(n)}  \stackrel{t}\to W_{j,k}^{(n)} .
\end{equation}
defined by
\[(g , p) \mapsto g \bullet p.
\]
where $g \bullet p$ is defined in Definition \ref{actdef}.
The multiplication
\[m:(G_{j,k}^{(n)} \times W_{j,k}^{(n)} )_{s} \times_{W_{j,k}^{(n)} } {}_{t}(G_{j,k}^{(n)} \times W_{j,k}^{(n)} ) \to  G_{j,k}^{(n)} \times W_{j,k}^{(n)}
\]
 is given by
\[m \left((g_{1},g_{2} \bullet p), (g_{2},p) \right) = (g_{1} \bullet_{p} g_{2}, p)
\]
where $g_{1} \bullet_{p} g_{2}$ is defined in definition \ref{dotp}.

The identity section is defined by \[e(p) = ( \text{id},p)\] and the inverse is defined by \[\iota(g,p) = (g^{(-1)_{p}},g \bullet p).\]
The associativity condition (\ref{cond1}) follows from (\ref{concludethat}). The conditions (\ref{cond2}), (\ref{cond4}) and (\ref{cond3}) follow from   (\ref{really}), (\ref{ofcourse}), and (\ref{annoying}).

We get an inverse system $\mathcal{G}_{j,k}^{(\bullet)}$ in the category of groupoids in schemes:

\begin{equation}\label{groupoidinvlim}\cdots \to  \mathcal{G}_{j,k}^{(3)} \to \mathcal{G}_{j,k}^{(2)}  \to \mathcal{G}_{j,k}^{(1)} \to \mathcal{G}_{j,k}^{(0)}.
\end{equation}
and the inverse limit is $\widetilde{\mathcal{G}_{j,k}}= \mathcal{G}_{j,k}^{(\infty)}$.

\subsection{The morphism defined via the big bundle $\mathbb{E}$}

The big bundle $\mathbb{E}$ defines a morphism of stacks from $W_{j,k}^{(n)}$ to $\mathfrak{M}_{j}(Z_{k}^{(n)})$ as follows.  Given an affine scheme $T$, we have a map
\[\varphi_{T} \colon \Hom(T,W_{j,k}^{(n)}) \to \Hom(T,\mathfrak{M}_{j}(Z_{k}^{(n)}) )
\]
\[f \mapsto (\text{id}, f)^{*}\mathbb{E}
\]
given by sending $f$ to the pullback of $\mathbb{E}$ via the map  \[(\text{id}, f):Z_{k}^{(n)} \times T \to Z_{k}^{(n)} \times W_{j,k}^{(n)}.\]
\begin{lemma} Let $E$ be a vector bundle on $Z_{k}^{(n)} \times T$ such that $\pi_{T *}(E\otimes \mathcal{O}(j)) $ is generated by global sections and $R^{1}\pi_{T *} (E\otimes \mathcal{O}(j))=0$.  Then there exists a Zariski open cover $\{U_{i}\}$ of $T$ such that the restriction of $E$ to each  $Z_{k}^{(n)} \times U_{i}$ is an extension of $\mathcal{O}(j+ \text{deg}(E))$ by $\mathcal{O}(-j)$, where $\text{deg}(E)$ is the degree of $E$ along the $Z_{k}^{(n)}$ fibers.
\end{lemma}
\begin{proof}
Choose a point $t \in T(\mathbb{C})$.  By assumption, there exists a global section
\[\sigma: \mathcal{O}_{T} \to \pi_{T *}(E\otimes \mathcal{O}(j))
\]
 of $\pi_{T *}(E\otimes \mathcal{O}(j)) $ which does not vanish at $t$.  Suppose that $\sigma$ vanishes along some divisor $D$ not containing $t$.  By adjointness, we can consider $\sigma$ as a map
\[\sigma: \mathcal{O} = \pi_{T}^{*}\mathcal{O}_{T} \to E\otimes\mathcal{O}(j)
\]
which vanishes along $Z_{k}^{(n)} \times D$.  Therefore, if we let $U' = T-D$ we have a short exact sequence
\[0 \to \mathcal{O}|_{Z_{k}^{(n)} \times U'} \to (E\otimes\mathcal{O}(j))|_{Z_{k}^{(n)} \times U'} \to \mathcal{S} \to 0.
\]
By considering the rank, $\mathcal{S}$ must be a line bundle on $Z_{k}^{(n)} \times U'$, and by shrinking $U'$ to some open set $U$ containing $t$ we may assume that this line bundle is a pullback from $Z_{k}^{(n)}$ and so it must be $\mathcal{O}(i)$ for some $i \in \mathbb{Z}$.  By consideration of the degree, we have $i = \text{deg}(E) + 2j$.   Tensoring with $\mathcal{O}(-j)$ we arrive at the required extension over $Z_{k}^{(n)} \times U$.
\end{proof}
\begin{lemma}
 The substacks \[\mathfrak{M}_{\leq j}(Z_{k}^{(n)})= \bigcup_{0\leq i \leq j} \mathfrak{M}_{i}(Z_{k}^{(n)})\] of $\mathfrak{M}(Z_{k}^{(n)})$ are given by
\[T \mapsto \left\{E \in \mathfrak{M}(Z_{k}^{(n)})(T)| \pi_{T *}(E\otimes \mathcal{O}(j)) \text{is generated by global sections and}\right. \]
\[ \left.\phantom{xxxxxxxxxxxxxXXXXXXXXXX} R^{1}\pi_{T *} (E\otimes \mathcal{O}(j))=0\right\}.
\]
\end{lemma}
\begin{proof}
  By Serre's theorem, $\mathfrak{M}(Z_{k}^{(n)})$ is covered by these open substacks.
In order to show the Lemma we can work locally in the site, and show the equivalence using the prestacks $\text{pre-}[\mathcal{G}_{j,k}^{(n)}]$. First suppose that $E$ is an extension of $\mathcal{O}(j)$ by $\mathcal{O}(-j)$.  Then $E\otimes \mathcal{O}(j)$ is an extension of $\mathcal{O}(2j)$ by $\mathcal{O}$. The resulting sequence on global sections is exact.
Both of the line bundles
$\mathcal{O}(2j)$ and  $\mathcal{O}$ are generated by their global sections,
 and the fact that $\pi_{T *}(E\otimes \mathcal{O}(j))$ is generated by its global sections follows.  However, $H^{1}(Z_{k}^{(n)}, \mathcal{O}(a))$ vanishes for $a \geq 0$ and therefore $R^{1}\pi_{T *} (E(j))$ vanishes.  Conversely, suppose that  $\pi_{T *}(E\otimes \mathcal{O}(j))$ is generated by global sections and $R^{1}\pi_{T *} (E\otimes \mathcal{O}(j))=0$.  The latter condition implies that for every point $t$, the splitting type of the restriction of $E$ to $Z_{k}^{(n)} \times\{t\}$ is less than or equal to $j$.  Consider a non-zero global section of $E\otimes \mathcal{O}(j)$.  If we restrict it to $\ell \times \{t\}$ for any geometric point $t$ it is a section of the bundle $\mathcal{O}(j+i) \oplus \mathcal{O}(j-i)$ on $\mathbb{P}^{1}$ for some $i$ such that $0 \leq i\leq j$ and therefore does not vanish, therefore it does not vanish on $\ell \times T $ and therefore it does not vanish on $Z_{k}^{(n)} \times T$.  We therefore get a constant rank,  injective map of sheaves $\mathcal{O}_{Z_{k}^{(n)}  \times T} \to E\otimes \mathcal{O}(j)$, and  a short exact sequence
\[0 \to \mathcal{O}_{Z_{k}^{(n)} \times T }(-j) \to E  \to \mathcal{O}_{T \times Z_{k}^{(n)} }(j) \to 0 \text{.}
\]
\end{proof}

\subsection{The universal bundle $\tilde{\mathcal{E}}$}\label{UniversalBundle}
We now construct the universal bundle on the groupoid 
\[Z_{k}^{(n)} \times \mathcal{G}_{j,k}^{(n)}.\]  The groupoid in question has atlas $Z_{k}^{(n)} \times W_{j,k}^{(n)}$ and relations $Z_{k}^{(n)} \times G_{j,k}^{(n)} \times W_{j,k}^{(n)}$.   We use the description of sheaves on groupoids in schemes given in subsection \ref{RevGrp}.  We start with the big bundle $\mathbb{E}$ on $Z_{k}^{(n)} \times W_{j,k}^{(n)}$ which was defined in Definition \ref{bigbundle}.
Consider the map in 
\[\Isom_{Z_{k}^{(n)} \times G_{j,k}^{(n)} \times W_{j,k}^{(n)}}((\text{id}_{Z_{k}^{(n)}},t)^{*} \mathbb{E}, (\text{id}_{Z_{k}^{(n)}},s)^{*} \mathbb{E})
\]
given by the pair
\[(A_{g}(p), B_{g}(p)) \in \Aut\left(U^{(n)} \times G_{j,k}^{(n)} \times W_{j,k}^{(n)}, \mathcal{O}^{\oplus 2} \right) \times \Aut \left(V^{(n)} \times G_{j,k}^{(n)} \times W_{j,k}^{(n)}, \mathcal{O}^{\oplus 2} \right)\]
which was defined in equations (\ref{Adef}) and (\ref{Bdef}).
We need to consider the pullbacks of the isomorphism to 
\[Z_{k}^{(n)} \times (G_{j,k}^{(n)} \times W_{j,k}^{(n)})_{s} \times_{W_{j,k}^{(n)}} {}_{t}(G_{j,k}^{(n)} \times W_{j,k}^{(n)})
\]
via the maps 
\[(\text{id}_{Z_{k}^{(n)}} ,m),(\text{id}_{Z_{k}^{(n)}} ,p_{1}),(\text{id}_{Z_{k}^{(n)}} ,p_{2})
\]
where $m,p_{1},p_{2}$ are the maps
\[(G_{j,k}^{(n)} \times W_{j,k}^{(n)})_{s} \times_{W_{j,k}^{(n)}} {}_{t}(G_{j,k}^{(n)} \times W_{j,k}^{(n)})\to G_{j,k}^{(n)}  \times W_{j,k}^{(n)}
\]
given by
\[m\left((g_{1},g_{2} \bullet p),(g_{2},p)\right)= (g_{1} \bullet_{p} g_{2},p)
\]
\[p_{1}\left((g_{1},g_{2} \bullet p),(g_{2},p)\right)= (g_{1},g_{2} \bullet p)
\]
and
\[p_{2}\left((g_{1},g_{2}\bullet p),(g_{2},p)\right)=(g_{2},p).
\]
 These pullbacks are described by the pairs of elements of
\[\Aut \left(U^{(n)} \times (G_{j,k}^{(n)} \times W_{j,k}^{(n)})_{s} \times_{W_{j,k}^{(n)}}  {}_{t}(G_{j,k}^{(n)} \times W_{j,k}^{(n)}) , \mathcal{O}^{\oplus 2} \right)
\]
and
\[\Aut \left(V^{(n)} \times (G_{j,k}^{(n)} \times W_{j,k}^{(n)})_{s} \times_{W_{j,k}^{(n)}}  {}_{t}(G_{j,k}^{(n)} \times W_{j,k}^{(n)}) , \mathcal{O}^{\oplus 2} \right)
\]
given by \[
(A_{g_{1} \bullet_{p} g_{2}}(p), B_{g_{1} \bullet_{p} g_{2}}(p)),\]
\[(A_{g_{1}}(g_{2} \bullet p), B_{g_{1}}( g_{2} \bullet p)),\] and \[(A_{g_{2}}(p), B_{g_{2}}( p))\] respectively.
Therefore identity (\ref{isomcond1}) follows from (\ref{defGrpStr}) while (\ref{isomcond2}) follows from (\ref{idenprop}) and 
consequently we have defined a vector bundle on the groupoid in accordance with the description in \ref{RevGrp}.

\subsection{The equivalence of stacks}

Let us first mention groupoid presentations in the case of line bundles.

The stack of line bundles on the $Z_{k}^{(n)}$ is equivalent to
\[\mathbb{Z} \times [\bullet/\mathcal{O}(Z_{k}^{(n)})^{\times}].
\]
For example when $k=1$, $n=\infty$ this stack is  equivalent to \[\mathbb{Z} \times [\bullet/\mathbb{C}[[x,y]]^{\times}].
\]
In section \ref{DefOfGrpd} we defined a groupoid in schemes
\[\mathcal{G}_{j,k}^{(n)}= (G_{j,k}^{(n)} \times W_{j,k}^{(n)},W_{j,k}^{(n)},m,e,\iota),
\] the associated pre-stack $\text{pre-}[\mathcal{G}_{j,k}^{(n)}]$ and the associated stack $[\mathcal{G}_{j,k}^{(n)}]$ on the fppf site.

\begin{theorem}
The natural map $W_{j,k}^{(n)} \to \mathfrak{M}_{j}(Z_{k}^{(n)})$ given by the big bundle $\mathbb{E}$ which was defined in Definition \ref{bigbundle}
 induces an isomorphism of stacks
\[[\mathcal{G}_{j,k}^{(n)}] \cong \mathfrak{M}_{j}(Z_{k}^{(n)}).
\]
Furthermore, there is a vector bundle
$$\begin{array}{c}
 [\tilde{\mathcal{E}}] \\
\downarrow  \\
Z^{(n)}_{k}\times [\mathcal{G}_{j,k}^{(n)}]
\end{array}$$
whose pullback to $Z^{(n)}_{k} \times W_{j,k}^{(n)}$ is the big bundle $\mathbb{E}$, and is identified via the above isomorphism with the universal bundle $\mathcal{E}$ on $Z^{(n)}_{k} \times \mathfrak{M}_{j}(Z_{k}^{(n)})$.
\end{theorem}
Here,
$[\mathcal{G}_{j,k}^{(n)}]$ is the stack associated to the groupoid
$\mathcal{G}_{j,k}^{(n)}$.  This association is reviewed in \ref{Grpds2Stacks}.

\begin{proof}
We will prove this theorem by first defining an morphism of stacks over the fppf site and then show that it is locally in the site an equivalence of categories.
Conisder the morphism of pre-stacks
\[\text{pre-}F: \text{pre-}[\mathcal{G}_{j,k}^{(n)}] \to \mathfrak{M}_{j}(Z_{k}^{(n)})
\]
by
\[\text{pre-}F_{T}(f)= 
(\text{id}_{Z_{k}^{(n)}},f)^{*}\tilde{\mathcal{E}}
\]
where $f$ is a morphism of groupoids from $T$ to $\mathcal{G}_{j,k}^{(n)}$.
Because $\mathfrak{M}_{j}(Z_{k}^{(n)})$ is already a stack over the fppf site, we get for free a morphism of the associated stacks  over the fppf site
\[F: [\mathcal{G}_{j,k}^{(n)}] \to \mathfrak{M}_{j}(Z_{k}^{(n)}).
\]
In order to show this is an equivalence we need only to show that it is locally an isomorphism.
Consider a vector bundle $E$ on $Z^{(n)}_{k} \times T$ for an affine $\mathbb{C}$ sheme $T$ and write it somehow (it does not matter how) as an extension of $\mathcal{O}(j)$ by $\mathcal{O}(-j)$ possibly after renaming $T$. Then by flat base change for the diagram
\[
\xymatrix{
Z^{(n)}_{k}\times T \ar[r]^{\pi_{T}} \ar[d]^{\pi_{Z^{(n)}_{k}}} & T \ar[d] \\
Z^{(n)}_{k} \ar[r] & \{ \cdot \}}
\]
and the Leray spectral sequence for $\pi_{Z^{(n)}_{k}}$ we have
\begin{equation} \label{LerayPres}
\begin{split}
& \Ext^{1}_{Z^{(n)}_{k} \times T}(\pi_{Z^{(n)}_{k}}^{*}\mathcal{O}(j), \pi_{Z^{(n)}_{k}}^{*}\mathcal{O}(-j)) = H^{1}(Z^{(n)}_{k} \times T, \pi_{Z^{(n)}_{k}}^{*}(\mathcal{O}(-2j)))
\\ &= H^{0}(T,R^{1}\pi_{T*} \pi_{Z^{(n)}_{k}}^{*}(\mathcal{O}(-2j))) = H^{0}(T,\mathcal{O}_{T} \otimes H^{1}(Z^{(n)}_{k},\mathcal{O}(-2j))) \\ &=
H^{0}(T,\mathcal{O}_{T} \otimes\ Ext^{1}_{Z^{(n)}_{k}}(\mathcal{O}(j),\mathcal{O}(-j))) = \Hom(T, W_{j,k}^{(n)}).
\end{split}
\end{equation}
We can conclude that choosing (locally in the test schemes) the structure of an extension gives maps from $T$ to the atlass of $\mathcal{G}_{j,k}^{(n)}$.  It remains to show that the ambiguity in such choices is given by maps from $T$ to the relations of $\mathcal{G}_{j,k}^{(n)}$.  Suppose we have two maps $p$ and $p'$ from $T$ to $W_{j,k}^{(n)}$.  We need to show that
  \[
 \Isom_{[\mathcal{G}_{j,k}^{(n)}](T)}(p,p') \cong \Isom_{Z^{(n)}_{k} \times T}((\text{id}_{Z^{(n)}_{k}},p)^{*}\mathbb{E} ,(\text{id}_{Z^{(n)}_{k}},p')^{*} \mathbb{E}).
 \]
We have already naturally identified these two sets in Lemma \ref{iso}.
\end{proof}

We can use some easy observations about the explicit presentation we have established to give some properties of the stacks $\mathfrak{M}_{j}(Z^{(n)}_{k})$.  First of all $G_{j,k}^{(n)}$ and $W^{(n)}_{j,k}$ are reduced, irreducible, affine algebraic varieties.  Notice that $s$ is a projection and the map $t$ factors as a Zariski open embedding followed by a projection
\[
\xymatrix{
R=G_{j,k}^{(n)} \times W_{j,k}^{(n)}  \ar[r]^{} \ar[rd]^{t} & \mathbb{E}\text{nd}_{Z_{k}^{(n)}}(\mathcal{O}(j) \oplus \mathcal{O}(-j)) \times W_{j,k}^{(n)} \ar[d] \\
 & A= W_{j,k}^{(n)}.}
\]
where the horizontal map is
\[(g,p) \mapsto (g,gp).
\]
The following could be concluded from the general construction of these stacks of vector bundles using Quot schemes due to Laumon and Moret-Bailly but we can give here a direct proof.
\begin{corollary}For every finite $n$, the stack $\mathfrak{M}_{j}(Z^{(n)}_{k})$ is an Artin stack.
\end{corollary}
\begin{proof}
When $n$ is finite then  $G_{j,k}^{(n)}$ and $W^{(n)}_{j,k}$ are smooth affine varieties of finite type.  By \cite{Lau} cor. 4.7, in order to conclude that it is an Artin stack, we need to show that $s$ and $t$ are flat and that the morphism
\[(s,t): R \to A \times A
\]
is separated and quasi-compact.  Since $n$ is finite, $s$ and $t$ are in fact smooth and therefore certainly flat.  Quasi-compactness is obvious since $R$ is quasi-compact.   To see that $(s,t)$ is separated we need to see that the induced diagonal
\begin{equation}\label{indDiag}
R \to R _{(s,t)}\times_{A \times A} {}_{(s,t)}R
\end{equation}
is closed.  Notice that $ R _{(s,t)}\times_{A \times A} {}_{(s,t)}R$ is a closed subvariety of
\[G_{j,k}^{(n)} \times G_{j,k}^{(n)} \times W^{(n)}_{j,k}
\] defined by the equation
\[g_{1}p=g_{2}p.
\] The image of the diagonal (\ref{indDiag}) is therefore closed, being just the intersection inside
\[G_{j,k}^{(n)} \times G_{j,k}^{(n)} \times W^{(n)}_{j,k}
\] of
\[R _{(s,t)}\times_{A \times A} {}_{(s,t)}R\] with the closed subvariety
\[\Delta_{G^{(n)}_{j,k}} \times W^{(n)}_{j,k}.\]
where 
\[\Delta_{G^{(n)}_{j,k}} \subset G^{(n)}_{j,k} \times G^{(n)}_{j,k}
\]
is the diagonal.
\end{proof}

\section{Applications}
In a forthcoming article \cite{BeG} we will use these groupoid presentations to calculate the space of deformations of the moduli stacks $\mathfrak{M}_{j}(Z_{k}^{(n)})$.  To do this one must calculate the cohomology of the tangent complex (thought of as a complex of coherent sheaves) on these stacks.  We then consider deformations of the $Z_{k}^{(n)}$.  These include both classical and non-commutative deformations of the type considered in \cite{BeBlPa2007} and \cite{To}.  By considering stacks of vector bundles over universal families of these deformations we get natural deformations of the stacks $\mathfrak{M}_{j}(Z_{k}^{(n)})$.  We investigate the corresponding map from deformations of  $Z_{k}^{(n)}$ to deformations of  $\mathfrak{M}_{j}(Z_{k}^{(n)})$.  This map is neither injective nor surjective.  Such maps are well understood for the case of curves (see for example \cite{NR}); whereas for surfaces such maps are only understood in a few special cases, such as  Mukai's \cite{Mu} description for the case of $K3$ surfaces.  In general such maps are quite mysterious for the case of surfaces. Thus, it is interesting to look at the question in the intermediate case of formal neighborhoods of curves inside surfaces.

Consider a proper algebraic surface $X$ over $\mathbb{C}$.   By attaching the stacks $\mathfrak{M}_{j}(\widetilde{Z_{k}})$ to $\mathfrak{M}(X)$ in the correct way one gets certain substacks $\mathfrak{M}_{j}(Y)$ of the stack of vector bundles on the blow up of $X$ at some point.  Consider the punctured space $Z_{k}^{\circ} = Z_{k}-\ell$ and the punctured formal neighborhood $\widetilde{Z_{k}}^{\circ}$ which is defined by 
\[\widetilde{Z_{k}}^{\circ} = \widetilde{Z_{k}} \times_{Z_{k}}  Z_{k}^{\circ}.
\] 
If $Y$ is any algebraic surface containing a rational curve $\ell$ with $\ell^{2}=-k$, $k>0$ then let $Y^{\circ}=Y-\ell$ and consider the commutative diagram
 
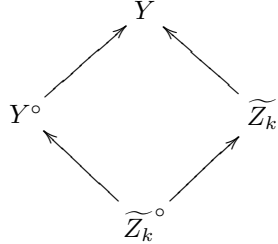
\[
\xymatrix{
 & Y   &  \\
Y^{\circ} \ar[ur] & & \widetilde{Z_{k}} \ar[ul] \\
& \widetilde{Z_{k}}^{\circ}  \ar[ul] \ar[ur]  &
}.
\] 
Let $\mathfrak{M}(Y)$ be the stack of all vector bundles of rank $2$ whose restriction to $\ell$ has first Chern class zero, while  $\mathfrak{M}(Y^{\circ})$ and $\mathfrak{M}(\widetilde{Z_{k}}^{\circ})$ are the stacks of all vector bundles of rank $2$ on $Y^{\circ}$ and $\widetilde{Z_{k}}^{\circ}$ respectively. 
By taking stacks of vector bundles and using the main theorem of \cite{BL2}, we get a fiber product diagram of stacks along with the substacks of splitting type $j$,
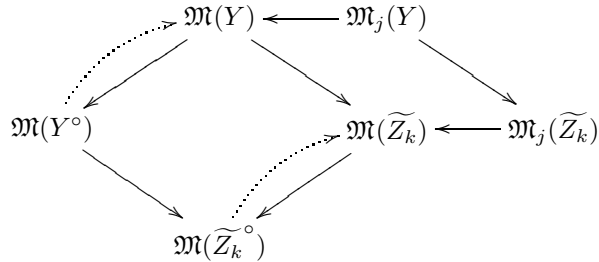
 \[
\xymatrix{
 & \mathfrak{M}_{}(Y)   \ar[dr] \ar[dl] &   \mathfrak{M}_{j}(Y) \ar[l] \ar[dr]  & \\
\mathfrak{M}(Y^{\circ}) \ar[dr]  \ar@/^1pc/@{.>}[ur]  & & \mathfrak{M}_{}(\widetilde{Z_{k}}) \ar[dl] & \mathfrak{M}_{j}(\widetilde{Z_{k}}) \ar[l]   \\
& \mathfrak{M}(\widetilde{Z_{k}}^{\circ})  \ar@/^1pc/@{.>}[ur]   & &
}
\]
consisting of the above diagram with the solid arrows only.    The dotted curved arrows going up here exist only in the case that $k=1$ and when the image of $\ell$ is a smooth point.  Suppose we are in this case and $\pi:Y \to X$ is the contraction of $\ell$.  Then the dotted arrows are sections of the arrows in the opposite direction and are given by extending a bundle from $Y^{\circ}=Y-\ell \cong X - \{x\}$ to a bundle in $\mathfrak{M}(X)$ by taking the double dual of its pushforward and then pulling back the bundle via $\pi$ to $Y$ (and similarly on the other side).  This diagram is an algebraic version of the holomorphic patching construction used in \cite{ADV} and can be used to get information about the relationship of $\mathfrak{M}_{j}(Y)$ and $\mathfrak{M}(Y-\ell)$ from the relationship of $\mathfrak{M}_{j}(\widetilde{Z_{k}})$  and $\mathfrak{M}(\widetilde{Z_{k}}^{\circ})$.  This version of patching using stacks is a much more powerful construction, in particular avoiding all-together the use of framings, hence eliminating the unnecessarily complicated issues of infinite dimensionality of the space of reframings of each individual bundle.  In this article we have focused on a description of  $\mathfrak{M}_{j}(\widetilde{Z_{k}})$.  The application to topological information will appear in a forthcoming article \cite{B2} where we use the groupoid presentation to compute homology, cohomology and homotopy groups of the stacks of bundles.

Another reason why using stacks of bundles is preferable for gluing purposes over the construction via framings is that framings (in the sense of trivialising sections) simply do not exist in general. For the case of a surface with a $-1$ line it turns out possible to add framings to all holomorphic bundles, that is, every bundle on $Z_1$ is trivial on $Z^{\circ}_1$, so one can consider pairs of bundles together with framings, and glue by identifying framings. However, for elements of ${\mathfrak M}_j(Z_k)$ only those satisfying $j=0 \mod k$ are trivial on $Z_k^\circ$. This argument becomes even more relevant if one considers curves inside threefolds. For instance over the resolved conifold 
$W_1=\Tot({\mathcal O}(-1)\oplus {\mathcal O}(-1)$
we can consider also rank 2 bundles with splitting $(j,-j)$ and define  stacks 
${\mathfrak M}_j(W_1)$ but here only the trivial bundle is frameable in the sense of \cite{ADV}.

\appendix
\section{Some cohomology groups}

The ring of global functions on $\widetilde{Z_{k}}$ is
\[ \mathcal{O}(\widetilde{Z_{k}})={\mathbb C}[[x_0,x_1, \dotsc, x_k]] \bigl/\sum_{i=0}^{k-2} \sum_{j=i+2}^{k}\bigl(x_i x_j - x_{i+1} x_{j-1} \bigr)
 \text{ ,} \]
and for $Z_{k}^{(n)}$ one gets $\mathcal{O}(Z_{k}^{(n)})= \mathcal{O}(Z_{k})/m^{n+1}$ where $m$ is the ideal $(x_{0},\dots, x_{k})$.  Note that here $x_i = z^i u$ in terms of the original coordinates on $U$ and $U^{(n)}$.
The zeroth cohomology is the torsion-free $\mathcal{O}(\widetilde{Z_{k}})$ module
\[H^{0}(\widetilde{Z_{k}},\mathcal{O}(s)) = \bigoplus_{ki+s-l \geq 0,l \geq 0,i \geq 0} \mathbb{C} z^{l}u^{i} \subset \mathcal{O}(\widetilde{U}).
\]
Similarly, we have the $\mathcal{O}(Z_{k}^{(n)})$ module 
\[H^{0}(Z_{k}^{(n)},\mathcal{O}(s)) = \bigoplus_{ki+s-l \geq 0,l \geq 0,n>i \geq 0} \mathbb{C} z^{l}u^{i} \subset \mathcal{O}(U^{(n)}).
\]

\begin{remark}\label{spectrum}
This is the $\mathbb{C}$ points of the spectrum of the polynomial algebra freely generated over $\mathbb{C}$ by variables indexed by pairs $(l,i)$ such that $ki+s-l \geq 0,l \geq 0,n>i \geq 0$.  We use that algebraic structure.
\end{remark}
It is also easy to see that $H^{1}(Z_{k}^{(n)},\mathcal{O}(s))$ vanishes for $s \geq 0$.

\section{The cohomology  spectral sequence of $\mathcal{H}om(E,F)$}\label{spectral}
Consider a scheme $Z$ covered by just two affine open sets $U_{1}$ and $U_{2}$ and two rank $2$ vector bundles $E$ and $F$ on $Z$ which trivialize on the $U_{i}$. Assume also that $H^{1}(Z, \mathcal{O})=0$.  The \v{C}ech complex for computing the cohomology of $\mathcal{H}om(E,F)$ on $Z$ looks like 
\[\Hom_{U_{1}}(E|_{U_{1}},F|_{U_{1}}) \oplus \Hom_{U_{2}}(E|_{U_{2}},F|_{U_{2}}) \to \Hom_{U_{1} \cap U_{2}}(E|_{U_{1} \cap U_{2}},F|_{U_{1} \cap U_{2}}).
\]
If we choose local trivializations for $E|_{U_{1}}, E|_{U_{2}}$ and $F|_{U_{1}},F|_{U_{2}}$ then the complex becomes 
\[\Hom_{U_{1}}(\mathcal{O}^{\oplus 2},\mathcal{O}^{\oplus 2}) \oplus \Hom_{U_{2}}(\mathcal{O}^{\oplus 2},\mathcal{O}^{\oplus 2}) \to \Hom_{U_{1} \cap U_{2}}(\mathcal{O}^{\oplus 2},\mathcal{O}^{\oplus 2})
\]
with differential 
\[(A,B) \mapsto G_{E}A-BG_{F}
\]
where $G_{E},G_{F}$ are the transition matrices of $E$ and $F$.  On the other hand suppose we know that $E$ and $F$ can be written on $Z$ as extensions of line bundles $L_{2}$ by $L_{1}$.   By choosing 
 local splittings the \v{C}ech complex becomes 
\[\End_{U_{1}}(\mathcal{O}^{\oplus 2}) \oplus \End_{U_{2}}(\mathcal{O}^{\oplus 2}) \stackrel{D_{1}} \to \End_{U_{1} \cap U_{2}}(\mathcal{O}^{\oplus 2})
\]
\[D_{1}(N_{1},N_{2})= \left(\begin{matrix} g_{1}& 0  \cr 0 & g_{2} \cr\end{matrix}\right)N_{1}-N_{2}\left(\begin{matrix} g_{1}& 0  \cr 0 & g_{2} \cr\end{matrix}\right),
\]
\[D_{2}(M_{1},M_{2}) = \left(\begin{matrix} g_{1}& p_{E}  \cr 0 & g_{2} \cr\end{matrix}\right)M_{1}-M_{2}\left(\begin{matrix} g_{1}& p_{F}  \cr 0 & g_{2} \cr\end{matrix}\right).
\]
\[\ker (D_{1}) \stackrel{\overline{D_{2}}}{\longrightarrow} \coker (D_{1})
\]

Let us compute the cohomology groups \[\ker(\overline{D_{2}}) =\Hom (E,F) \cong H^{0}(X,\mathcal{H}om(E,F))\] and \[\coker(\overline{D_{2}})= \Ext^{1}(E,F) \cong  H^{1}(X,\mathcal{H}om(E,F))\] in terms of the extension and cohomology groups of the $L_{i}$.  
The  filtration on $\mathcal{H}om(E,F)$ reads
\[0 \subset \mathcal{H}om(L_{2},L_{1})  \subset\mathcal{H}om(E,L_{1})  + \mathcal{H}om(L_{2},F)  \subset \mathcal{H}om(E,F)
\]
with associated graded pieces  $\mathcal{H}om(L_{2},L_{1}) $,  $\mathcal{E}nd(L_{1}) \oplus\mathcal{E}nd(L_{2})$, and $\mathcal{H}om(L_{1},L_{2}) $.
The associated spectral sequence computing the cohomology $\mathcal{H}om(E,F)$ has an $E_{1}$ term which looks like
\def\csg#1{\save[].[dddddrrr]!C*+<5pc,0pc>[F-,]\frm{}\restore}
\[
\xymatrix@=.4pc{
q=2 & \csg1
\vdots &
\vdots &
\vdots & \vdots   \\
q=1 &
\vdots  &
\vdots &
\vdots & \vdots  \\
q= 0 & \Hom(L_{1},L_{2})  &
 0  &  \vdots
 & 0
  \\
q=-1 &    &
\End(L_{1}) \oplus \End(L_{2})   &
\Ext^{1}(L_{2}, L_{1})&0    \\
q=-2 &
&   &
\Hom(L_{2},L_{1}) & 0  \\
q=-3 & & & &0  \\
& p= 0 & p= 1 & p=2 & p=3
}
\]

The $E_{2}$ term  looks like
\def\csg#1{\save[].[dddddrrr]!C*+<5pc,0pc>[F-,]\frm{}\restore}
\[
\xymatrix@=.7pc{
q=2 & \csg1
\vdots &
\vdots &
\vdots & \vdots   \\
q=1 &
\vdots  &
\vdots &
\vdots & \vdots  \\
q= 0 &\Hom(L_{1},L_{2}) &
0 &  \vdots
 & 0
  \\
q=-1 &    &
\text{ker}(d_{1}^{1,-1})  &
\text{coker}(d_{1}^{1,-1}) &0    \\
q=-2 &
&   &
\Hom(L_{2},L_{1}) & 0  \\
q=-3 & & & &0  \\
& p= 0 & p= 1 & p=2 & p=3
}
\]

 The $E_{3}$ term  looks like
\def\csg#1{\save[].[dddddrrr]!C*+<4pc,0pc>[F-,]\frm{}\restore}
\[
\xymatrix@=.7pc{
q=2 & \csg1
\vdots &
\vdots &
\vdots & \vdots   \\
q=1 &
\vdots  &
\vdots &
\vdots & \vdots  \\
q= 0 & \text{ker}(d_{2}^{0,0})  &
0 &  \vdots
 & 0
  \\
q=-1 &    &
\text{ker}(d_{1}^{1,-1})  &
\text{coker}(d_{1}^{1,-1}) /\text{im}(d_{2}^{0,0}) &0    \\
q=-2 &
&   &
\Hom(L_{2},L_{1}) & 0  \\
q=-3 & & & &0  \\
& p= 0 & p= 1 & p=2 & p=3
}
\]
   The first differential we consider is
\[H^{0}(X,\mathcal{O})^{\oplus 2}=\End(L_{1}) \oplus \End(L_{2}) \stackrel{d_{1}^{1,-1}}{\to}  \Ext^{1}(L_{2},L_{1}).
\]
It is the connecting map for the cohomology of the short exact sequence
\[0 \to \mathcal{H}om(L_{2},L_{1}) \to \mathcal{H}om(L_{2},F)+\mathcal{H}om(E,L_{1})   \to \mathcal{E}nd(L_{1}) \oplus  \mathcal{E}nd(L_{2}) \to 0
\]

Consider the induced filtration on $\Hom(E,F)$ given by
\[0 \subset \Hom(L_{2},L_{1}) \subset \Hom(E,L_{1}) +\Hom(L_{2},F) \subset \Hom(E,F).
\]
One has
\[\frac{\Hom (E,F)}{\Hom(E,L_{1}) +\Hom(L_{2},F)} \cong \text{ker}(d_{2}^{0,0}) \subset\Hom(L_{1},L_{2}),
\]
and
\[\frac{\Hom(E,L_{1}) +\Hom(L_{2},F)}{\Hom(L_{2},L_{1}) } \cong \text{ker}(d_{1}^{1,-1}) \subset H^{0}(X,\mathcal{O})^{\oplus 2}.
\]

For any choices of splittings
\[\Hom (E,F) \stackrel{\psi}{\leftarrow} \text{ker}(d_{2}^{0,0}) \subset \Hom(L_{1},L_{2})
\]
and
\[ \Hom(E,L_{1}) +\Hom(L_{2},F) \stackrel{\phi}{\leftarrow} \text{ker}(d_{1}^{1,-1}) \subset H^{0}(X,\mathcal{O})^{\oplus 2}
\]
we get a decomposition
\begin{equation}\label{ssdecomp}
\Hom (E,F) =\Hom(L_{2},L_{1}) \oplus \phi(\text{ker}(d_{1}^{1,-1})) \oplus \psi(\text{ker}(d_{2}^{0,0})).
\end{equation}
We record formulas for $d_{1}^{1,-1}$ and $d_{2}^{0,0}$ in the case that $X = Z_{k}^{(n)} \times T$ for some affine scheme $T$, $L_{1}=\mathcal{O}(-j)$, $L_{2}=\mathcal{O}(j)$, $E=E_{p}$, $F=E_{p'}$.  
$$d_1^{1,-1}:H^{0}(X,(L_{1} \otimes L_{1}^{\vee}) \oplus (L_{2} \otimes L_{2}^{\vee}))
\rightarrow
\Ext^{1}(L_{2},L_{1})$$
We compute
\[\left(\begin{matrix} z^{j} & p'  \cr 0 & z^{-j} \cr\end{matrix}\right) \left(\begin{matrix} \ua & 0  \cr 0 & \ud \cr\end{matrix}\right)  - \left(\begin{matrix} \ua & 0  \cr 0 & \ud \cr\end{matrix}\right)\left(\begin{matrix} z^{j} & p  \cr 0 & z^{-j} \cr\end{matrix}\right)= \left(\begin{matrix}0 & \ud p'  - \ua p \cr 0 & 0 \cr\end{matrix}\right) .
\]
Therefore the element of $\Ext^{1}(L_{2},L_{1})$ to which the pair $(\ua,\ud)$ maps is represented by $(\ud p'  - \ua p)|_{(U^{(n)} \cap V^{(n)}) \times T}$.The differential
\[d_{1}^{1,-1}: H^{0}(X, \mathcal{O}^{\oplus 2}) \to \Ext^{1}(\mathcal{O}(j),\mathcal{O}(-j))
\]
\[(\ua, \ud) \mapsto \ud p' - \ua p.
\]
In order to write down the next differential
\[d_{2}^{0,0}:\Hom(\mathcal{O}(-j),\mathcal{O}(j)) \to \Ext^{1}(\mathcal{O}(j),\mathcal{O}(-j))/\text{image}(d_{1}^{1,-1})
\text{,}\]
 we  choose regular functions $\alpha_{U}, \delta_{U}$ on $U$ and  $\alpha_{V}, \delta_{V}$ on $V$ such that

\[-z^{-j}p'\uc_{U} = \alpha_{U}- \alpha_{V}
\]
\[z^{j}p\uc_{U}= \delta_{U}-\delta_{V}
\]
so

\[d_{2}^{0,0}(\uc ) = \delta_{U} p'-\alpha_{V}p.
\]


\vspace{3mm}






\begin{thebibliography}{WWW}

\bibitem[BG1]{PAMS} Ballico, E.\ and Gasparim, E. {\it Numerical invariants
for vector bundles on blow-ups}, Proc.\ Amer.\ Math.\ Soc. \textbf{130}
(2002) 23--32.

\bibitem[BG2]{FM} Ballico, E.\ and Gasparim, E. {\it Vector bundles on a
neighborhood of an exceptional curve and elementary transformations},
Forum Math. \textbf{15} (2003) 115--122.

\bibitem[BG3]{RMJM} Ballico, E.\ and Gasparim, E. {\it Vector bundles on a
formal neighborhood of a curve in a surface}, Rocky Mountain J.\ Math.
\textbf{30} (2000) 795--814.

\bibitem[BGK]{BGK} Ballico, E., Gasparim, E. and K\"oppe, T. {\it Vector bundles
near negative curves: moduli and local Euler characteristic},
Comm. Algebra \textbf{37} no. 8  (2009) 2688--2713.


\bibitem[B1]{Be2006} Ben-Bassat, O. {\it
 Twisting Derived Equivalences,}
Transactions of the American Mathematical Society
\textbf{361} no. 10 (2009) 5469--5504.

\bibitem[B2]{B2}
Ben-Bassat, O. {\it The topology of stacks of vector bundles on some curves and surfaces}, in preparation.

\bibitem[BBP]{BeBlPa2007}
Ben-Bassat, O.,  Block, J.  and Pantev, T. {\it Non-commutative tori and 
Fourier--Mukai duality,}
Compositio Mathematica, 143, 423-475 (2007).


\bibitem[BeG]{BeG} Ben-Bassat, O.\ and Gasparim, E. {\it Exotic deformations of open surfaces and their stacks of vector bundles}, in preparation.
 

\bibitem[BL1]{BL1} Beauville, A. and Laszlo, Y. {\it Conformal blocks and generalized theta functions}, Comm.\ Math.\ Phys. \textbf{164} (1994) 385--419.

\bibitem[BL2]{BL2} Beauville, A. and Laszlo, Y. {\it Un lemme de descente}, C. R. Acad. Sci. Paris Ser. I Math. \textbf{320} no. 3, (1995) 335--340.

\bibitem[Co]{Co} Cohn, P. M. {\it Some remarks on projective free rings}, Algebra Universalis, no. 2, (2003) 159--164.

\bibitem[Dr]{Dr} Dr\'ezet, J.-M. {\it Exotic fine moduli spaces of coherent sheaves}, Algebraic cycles, sheaves, shtukas, and moduli; Impanga lecture notes, Trends in Mathematics, Birkhäuser (2008) 21--32.

\bibitem[G1]{CA1} Gasparim, E. {\it Holomorphic bundles on $\mathcal{O}(-k)$ are algebraic}, Comm.\ Algebra \textbf{25} (1997) 3001--3009.

\bibitem[G2]{JA} Gasparim, E. {\it Rank two bundles on the blow-up of \
$\mathbb C^2$}, J.\ Algebra \textbf{199} (1998) 581--590.

\bibitem[G3]{CA2} Gasparim, E. {\it Chern classes of bundles on blown-up surfaces}, Comm.\ Algebra \textbf{28} (2000) 4912--4926.

\bibitem[G4]{ADV} Gasparim, E. {\it The Atiyah--Jones conjecture for rational
surfaces}, Advances Math. \textbf{218} (2008) 1027--1050.

\bibitem[GK]{GK} Gasparim, E. and K\"oppe, T. {\it
Sheaves on singular varieties},  J. Singularities \textbf{2}
 (2010) 56-66,
Proceedings of Singularities in Aarhus, August 2009.

\bibitem[GO]{GO} Gasparim, E.\ and Ontaneda, P. {\it Three applications of instanton numbers}, Comm.\ Math.\ Phys. \textbf{270} no.\ 1 (2007) 1--12.

\bibitem[H]{HA} Hartshorne, R. {\it Algebraic Geometry}, Graduate Texts in Mathematics, Springer Verlag (1977).



\bibitem[L]{LA} Lange, H. {\it Universal Families of Extensions}, J. Algebra \textbf{83} no. 1 (1983) 101--112.

\bibitem[Lau]{Lau} Laumon, G., {\it Champs alg\'ebriques,} Prepublications {\textbf 88-33}, U. Paris-Sud (1988).

\bibitem[Mu]{Mu} Mukai, S. {\it On the moduli spaces of bundles on K3 surfaces, I}, "Vector bundles on Algebraic Varieties", Tata Institute of Fundamental Research, Bombay, (1984) 

\bibitem[NR]{NR} Narasimhan, M.S. and Ramanan, S. {\it Deformations of the moduli space of vector bundles over an algebraic curve}, Ann. of Math. {\textbf 101} (1975), 391--417.
\bibitem[Qi]{Qi} Quillen, D. {\it Projective modules over polynomial rings}, Invent. Math. {\bf 36} (1976) 166--172. 

\bibitem[Ry]{Ry} Rydh, N. {\it Noetherian approximation of algebraic spaces and stacks},  {\tt http://arxiv.org/abs/0904.0227}

\bibitem[Se] {Se} Seshadri, C. S. {\it Triviality of vector bundles over the affine space $k^2$}, Proc. Nat. Acad. Sci. U.S.A. {\bf 44} (1958) 456-458.

\bibitem[Sw]{Sw} Swan, R. {\it Projective modules over Laurent polynomial rings}, Trans. Amer. Math. Soc. \textbf{237} (1978)  111--120.

\bibitem[Su]{Su} Suslin, A. A. {\it Projective modules over polynomial rings are free}, Dokl. Acad. Nauk. SSSR, {\bf 229} no. 5, (1976) 1063--1066. 

\bibitem[To]{To} Toda, Y. {\it Deformations and Fourier-Mukai transforms}, J. Differential Geom. {\textbf 81} no. 1 (2009)  197--224. 

\end{thebibliography}
\end{document}